\documentclass[a4paper,reqno, 12pt]{amsart}

\usepackage{amsfonts}
\usepackage{amsmath}
\usepackage{amsthm}
\usepackage{amssymb}
\usepackage{hhline}
 \usepackage[pdftex]{graphicx}

\usepackage{hyperref}
\usepackage[final]{pdfpages}
\usepackage{subfigure}  % package for subfigure formatting
\usepackage{caption}
\newcommand{\RR}{{\mathbb{R}}}
\newcommand{\NN}{{\mathbb{N}}}

\newcommand{\C}{{\mathcal{C}}}
\newcommand{\I}{{\mathcal{I}}}
\newcommand{\J}{{\mathcal{J}}}
\newcommand{\B}{{\mathcal{B}}}

\newcommand{\R}{\mathcal{R}}

\newcommand{\Area}{\operatorname{Area}}

\newtheorem{lemma}{Lemma}
\newtheorem{conjecture}{Conjecture}
\newtheorem{theorem}{Theorem}%[section]
\newtheorem{corollary}{Corollary}%[section]

\begin{document}
\author{Cristian Cobeli, Yves Gallot, Pieter Moree and Alexandru Zaharescu
}

 \address{Cristian Cobeli,
 Institute of Mathematics  "Simion Stoilow" of the Romanian Academy,
 P. O. Box \mbox{1-764}, Bucharest 70700,
 Romania.}
 \email{cristian.cobeli@imar.ro}

 \address{Yves Gallot,
12 bis rue Perrey, 31400 Toulouse, France.}
 \email{galloty@orange.fr}

 \address{Pieter Moree,
Max-Planck-Institut f\"ur Mathematik, Vivatsgasse 7, D--53111 Bonn, Germany.}
 \email{moree@mpim-bonn.mpg.de}

 \address{Alexandru Zaharescu,
Department of Mathematics, University of Illinois at Urbana-Champaign, 
273 Altgeld Hall,
\mbox{MC-382}, 1409 W. Green Street, Urbana, Illinois 61801, USA 
and
 Institute of Mathematics  "Simion Stoilow" of the Romanian Academy,
 P. O. Box \mbox{1-764}, Bucharest 70700,
 Romania.}
 \email{zaharesc@math.uiuc.edu}

\subjclass[2000]{
Primary 11T22, %Cyclotomy
Secondary 11L05 %Gauss and Kloosterman sums; generalizations
%Primary 11N37, % Asymptotic results on arithmetic functions
%Secondary 11B57 %Farey sequences; the sequences 1
%11A07; %Congruences; primitive roots; residue systems
%Secondary
%11B50, %Sequences (mod m)
%11L07 %Estimates on exponential sums 
%11T23 %Exponential sums 
}
\thanks{Key Words and Phrases: Cyclotomic coefficients, Sister Beiter conjecture, Modular inverses, Kloosterman sums}

\title[A Tale of Cyclotomic Coefficients and
Modular Inverses]{Sister Beiter and Kloosterman:  A Tale of Cyclotomic Coefficients and
Modular Inverses}
%\maketitle
%\begin{center}
%\LARGE
%Sister Beiter and Kloosterman:  A Tale of Cyclotomic Coefficients and
%Modular Inverses
%\end{center}

%\begin{center}
%Cristian Cobeli, Yves Gallot, Pieter Moree and Alexandru Zaharescu
%\end{center}

%\authors{}

%\bigskip
\begin{abstract}
For a fixed prime $p$, the maximum coefficient (in absolute value) $M(p)$ of 
the cyclotomic polynomial $\Phi_{pqr}(x)$, where $r$ and $q$ are free primes
satisfying $r>q>p$ exists. Sister Beiter conjectured in 1968 that $M(p)\le(p+1)/2$.
In 2009 Gallot and Moree showed that $M(p)\ge 2p(1-\epsilon)/3$ for every $p$ sufficiently large.
In this article Kloosterman sums (`cloister man sums') and other tools from the
distribution of modular inverses are applied to quantify the abundancy of
counter-examples to Sister Beiter's conjecture and sharpen the above lower bound for $M(p)$.
\end{abstract}
\maketitle

\section{Introduction}
\label{nul}
The $n$-th cyclotomic polynomial $\Phi_n(x)$ is defined by
$$\Phi_n(x)=\prod_{1\le j\le n\atop (j,n)=1}(x-\zeta_n^j)=\sum_{k=0}^{\infty}a_n(k)x^k,$$ with
$\zeta_n$ a $n$-th primitive root of unity (one can take $\zeta_n=e^{2\pi i/n}$).
It has degree $\varphi(n)$, with $\varphi$ Euler's totient function. 
We write 
$A(n)=\max\{|a_n(k)|:k\ge 0\}$, and this quantity is called the height of $\Phi_n(x)$. It is easy
to see that $A(n)=A(N)$, with $N=\prod_{p|n,~p>2}p$ the odd squarefree kernel. In deriving this
one uses the observation that if $n$ is odd, then $A(2n)=A(n)$. If $n$ has at most two distinct odd prime factors, 
then $A(n)=1$. If $A(n)>1$, then we necessarily must have
that $n$ has at least three distinct odd prime factors. Thus for $n<105$ we have $A(n)=1$. It turns
out that $A(3\cdot 5\cdot 7)=2$ with $a_{105}(7)=-2$. Thus the easiest case where we can expect non-trivial
behavior of the coefficients of $\Phi_n(x)$ is the ternary case, where $n=pqr$, with $2<p<q<r$ odd
primes.
It is for this reason that in this paper we will be mainly interested in the behavior of
coefficients of ternary cyclotomic polynomials.\\
\indent If $n$ is a prime, then we have $\Phi_n(x)=1+x+\cdots+x^{n-1}$. Already if $n=pq$ consists of two prime factors and is
odd, modular
inverses come into the picture.
In this binary case the coefficients are computed
in the following lemma. {}For a proof
see e.g. Lam and Leung~\cite{LL} or Thangadurai~\cite{Thanga}.
\begin{lemma}
\label{binary}
Let $p<q$ be odd primes. Let $\rho$ and $\sigma$ be the (unique) non-negative
integers for which $1+pq=\rho p+ \sigma q$.
Let $0\le m<pq$. Then either $m=\alpha_1p+\beta_1q$ or $m=\alpha_1p+\beta_1q-pq$
with $0\le \alpha_1\le q-1$ the unique integer such that $\alpha_1 p\equiv m({\rm mod~}q)$
and $0\le \beta_1\le p-1$ the unique integer such that $\beta_1 q\equiv m({\rm mod~}p)$.
The cyclotomic coefficient $a_{pq}(m)$ equals
$$
\begin{cases}
1 & \mbox{if } m=\alpha_1p+\beta_1q \mbox{ with } 0\le \alpha_1\le \rho-1,~0\le \beta_1\le \sigma-1;\\ 
-1 & \mbox{if } m=\alpha_1p+\beta_1q-pq \mbox{ with } \rho\le \alpha_1\le q-1,~\sigma\le \beta_1\le p-1;\\  
0 & otherwise.
\end{cases}
$$
\end{lemma}
Note that $\rho$ is merely the modular inverse of $p$ modulo $q$ and $\sigma$ is the modular inverse
of $q$ modulo $p$. In the ternary case Kaplan's lemma~\cite{Kaplan} can be used to express a ternary cyclotomic coefficient
into a sum of binary ones. 
It is thus not surprising that also in the ternary case modular inverses make their appearance. We will give
some examples of this.

 Let $\overline q$ and $\overline r$, $0<\overline q,\overline r<p$ be the inverses of $q$ and $r$ modulo $p$ respectively.
Set $a=\min(\overline q,\overline r,p-\overline q,p-\overline r)$. 
Put $b=\max(\min(\overline q,p-\overline q),\min(\overline r,p-\overline r))$.  Note that $b\ge a$.
Bzd{\c e}ga~\cite{BZ} showed that
\begin{equation}
\label{barbound}
A(pqr)\le \min(2a+b,p-b).
\end{equation}
It is easy to show from this estimate that $A(pqr)<3p/4$ (see, e.g., Section 3 of Gallot et al.~\cite{GMW}). Notice that
this bound does not depend on the two largest prime factors of $n$. Indeed, for an arbitrary $n$ it was shown 
by Justin~\cite{J} 
and independently by Felsch and Schmidt~\cite{FS}
that there is an upper bound for $A(n)$ that does not depend on the largest and second largest prime factor of $n$.  
Thus for a fixed prime $p$ the maximum 
$$M(p):=\max\{A(pqr):p<q<r\},$$ 
where $q,r$ range over all the primes satisfying $p<q<r$, exists.
The major open problem involving ternary cyclotomic coefficients, is to find a finite procedure to determine $M(p)$.\\
\indent H. M\"oller~\cite{Moeller} gave a construction showing that $M(p)\ge (p+1)/2$ for $p>5$. On the other hand,
in 1968 Sister Marion Beiter~\cite{Beiter-1} had conjectured (a conjecture she repeated
in 1971~\cite{Beiter-2}) that $M(p)\le (p+1)/2$ and shown
that $M(3)=2$~\cite{Beiter-3}, which on combining leads to the conjecture that $M(p)=(p+1)/2$ for $p>2$. The bound
of M\"oller together with $M(5)\le 3$ (established independently by Beiter~\cite{Beiter-2} and
Bloom~\cite{Bloom}) shows that $M(5)=3$.
Zhao and Zhang~\cite{ZZ} showed that $M(7)=4$. Thus Beiter's conjecture holds true for $p\le 7$.
However, work
of Gallot and Moree~\cite{GM} has made clear that the true behavior of $M(p)$
is much more complicated than suggested by Beiter's conjecture.  
Theorem~\ref{main}, the main result of~\cite{GM}, produces counter-examples to Sister Beiter's 
conjecture. The goal of this paper is to investigate the abundance of these counter-examples
using techniques from the study of the distribution of modular inverses (for a survey, see, e.g.,
Shparlinski~\cite{Shparlinski}). These techniques involve Kloosterman sums $K(a,b;p)$. Recall that
for a prime $p$ the Kloosterman sum $K(a,b;p)$ is defined as
$$K(a,b;p)=\sum_{1\le x\le p-1}e^{2\pi i (ax+b{\overline x})/p},$$
where ${\overline x}$ denotes an inverse of $x$ modulo $p$.
By a fundamental result of Weil~\cite{W} we have that 
\begin{equation}
\label{weilie}
|K(a,b;p)|\le 2\sqrt{p}.
\end{equation}
\begin{theorem}
\label{main}
Let $p$ be a prime. Given an $1\le \beta\le p-1$, we let $\overline \beta$ be the unique
integer $1\le \overline \beta\le p-1$ with $\beta \overline \beta\equiv 1({\rm mod~}p)$. \\
\indent Let
${\B}_{-}(p)$ be the set of integers $\beta$ satisfying
\begin{equation}
\label{verg2}
1\le \beta\le {p-3\over 2},~p\le \beta+2\overline \beta+1,~\beta>\overline \beta.
\end{equation}
For every prime $q\equiv \beta({\rm mod~}p)$ with 
$q>q_{-}(p)$ and
$\beta\in {\B}_{-}(p)$, there exists a prime $r_{-}>q$ and 
an integer $n_{-}$ such that $a_{pqr_{-}}(n_{-})=\beta_{-}-p$, where $q_{-}(p),r_{-}$ and $n_{-}$
can be explicitly given.\\
\indent Let
${\B}_+(p)$ be the set of integers $\beta$ satisfying
\begin{equation}
\label{verg1}
1\le \beta\le {p-3\over 2},~\beta+\overline \beta\ge p,~\overline \beta\le 2\beta,
\end{equation}
For every prime
$q\equiv \beta({\rm mod~}p)$ with $q>q_{+}(p)$ and $\beta \in 
{\B}_{+}(p)$ there exists a prime $r_{+}>q$ and 
an integer $n_{+}$ such that $a_{pqr_{+}}(n_{+})=p-\beta$, where $q_{+}(p),r_{+}$ and $n_{+}$
can be explicitly given. In case $\beta\in {\B}_{+}(p)$ and $\beta+\overline \beta=p$,
then $A(pqr_{+})=p-\beta$.
\end{theorem}
\begin{corollary}
\label{coo}
Put ${\B}(p)={\B}_{-}(p)\cup {\B}_{+}(p)$. If
${\B}(p)$ is non-empty, then
$$M(p)\ge p-{\min}\{{\B}(p)\}>{p+1\over 2},$$
an so Beiter's conjecture is false for the prime $p$.
\end{corollary}
The explicit values of $q_{-}(p),r_{-},n_{-},q_{+}(p),r_{+}$ and $n_{+}$ will
be of no concern to us here. For these the reader is referred to
Theorems 10 and 11 in~\cite{GM}.\\
\indent We like to remark that the sets $\B_{\rm \pm}(p)$ are not merely `figments of the proof
of Theorem~\ref{main}'.  Similar (but not equal) sets were independently found by E. Ro\c su in her construction of
`Non-Beiter ternary cyclotomic polynomials with an optimally large set of coefficients', see~\cite{MR}.\\
\indent To fully exploit the power of Theorem~\ref{main}, one needs information
on the sets ${\B}_{-}(p), {\B}_{+}(p)$ and ${\B}(p)$. By an 
elementary method in~\cite{GM} the following information on ${\B}(p)$ was
deduced, which in combination with Theorem~\ref{main} shows that Beiter's
conjecture is false for every $p\ge 11$.
\begin{lemma}
\label{uno}
For $p\ge 11$, ${\B}(p)$ is non-empty and
$\max\{{\B}(p)\}={(p-3)/2}$.
\end{lemma}
\begin{proof}
Consider $\beta=(p-3)/2$. If $p\equiv 1({\rm mod~}3)$, then
$\overline \beta=2(p-1)/3$ and one checks that $\beta\in {\B}_{+}(p)$. 
If $p\equiv 2({\rm mod~}3)$, then
$\overline \beta=(p-2)/3$ and one checks that $\beta\in {\B}_{-}(p)$.
\end{proof}

Showing the non-emptiness of $\B(p)$ for $p\ge 11$ is thus almost trivial. 
Estimating its cardinality is rather more challenging and this is were the Kloosterman sums come in.
%%%%%%%%%%%%%%%%%%%%%%%%%%%%%
\begin{theorem}\label{Theorem1}
For any prime number $p$,
\begin{align}
\bigg|\#\B_{-}(p)-\frac{p}{48}\bigg|&\le 12\;p^{3/4}\log p\,,\notag \\
\bigg|\#\B_{+}(p)-\frac{p}{24}\bigg|&\le 12\;p^{3/4}\log p\,,\\ 
\bigg|\#\B(p)-\frac{p}{16}\bigg|&\le 24\;p^{3/4}\log p\, \notag.
\end{align}
\end{theorem}

\bigskip
%==================================================
\indent It was shown~\cite[Proposition 4]{GM}, working with explicit inverses modulo
$p$, that if $e\ge 1$ and $p$ are such (with $N=2^{2e+1}$) that if
$$\epsilon>0,~N>{1\over 3\epsilon}+3,~p>{N^2\over 2}-9{\rm ~and~}p\equiv N-9({\rm mod~}3N),$$
then min$\{{\B}_{+}(p)\}<{p\over 3}(1+\epsilon)$ and 
hence $M(p)>({2\over 3}-\epsilon)p$. An easy application of Lemma~\ref{Lemma1}  below (see the proof of Theorem 6 of~\cite{GM})
yields the stronger result that
\begin{equation}\label{e2}
\frac23p(1-\epsilon)\le M(p),
\end{equation}
for {\em every} prime $p$ large enough.\\
\indent If there are $p$ with $M(p)>2p/3$, then Theorem~\ref{main} does not allow to find them, since $\min\{\B(p)\}\ge p/3$.
On this basis
and extensive numerical experiments by Gallot, the following 
Corrected Beiter Conjecture (Conjecture 3 from~\cite{GM}) can be made:
\begin{equation}\label{e1}
M(p) \le \frac{2p}3.
\end{equation}
If true, this conjecture would place $M(p)$ in a rather short interval of size $\epsilon p$. \\
\indent A natural question that arises would be to see
how much one can shorten this interval by improving on the lower bound in \eqref{e2}.
We will establish the following result.
\begin{theorem}$~$\\ \label{driedrie}
{\rm 1)} We have
\begin{equation}\label{e3}
M(p) >  \frac{2p}3 - 3\, p^{3/4}\log p.
\end{equation}
{\rm 2)} For an infinite class of prime numbers $p$ we have
\begin{equation}\label{e4}
M(p) >  \frac{2p}3 - c_1 \sqrt{p},
\end{equation}
with $c_1$ a positive constant.
\end{theorem}

Given fixed primes $2<p<q$, put 
$$M(p;q):=\max\{A(pqr):p<q<r\},$$ 
where $r$ ranges over all the primes $>q$. There is a finite procedure to determine
$M(p;q)$. We say that a function is ultimately constant on an infinite sequences of integers, if it
takes on the same value for all sufficiently large elements in the sequence. The study of
$M(p;q)$ was initiated by Gallot et al.~\cite{GMW}. The main conjecture is that given
a prime $p$, there exists a modulus $d_p$, such that $M(p;q)$ is ultimately constant
on every primitive residue class modulo $d_p$. This would imply that
$$\delta_p=\lim_{x\rightarrow \infty}{\#\{p<q\le x: M(p;q)> (p+1)/2\}\over \pi(x)},$$
exists and is rational (by the prime number theorem for arithmetic progressions). 
Here as usual $\pi(x)$ denotes the number of primes $p\le x$,
Put
$${\underline \delta}_p=\lim_{x\rightarrow \infty}\inf{\#\{p<q\le x: M(p;q)> (p+1)/2\}\over \pi(x)}.$$
By Theorem~\ref{main}, Lemma~\ref{uno} and the prime number theorem for arithmetic progressions
${\underline \delta}_p$ is positive for $p\ge 11$. We will establish the following result, which in conjunction with
Lemma~\ref{uno} implies that there is a positive
constant $c_2$ such that ${\underline \delta}_p\ge c_2$ for every prime $p\ge 11$. 
\begin{theorem}
We have
$${\underline \delta}_p\ge {\# \B(p)\over p-1}{\rm ~and~}\lim \inf_{p\rightarrow \infty}{\underline \delta}_p\ge {1\over 16}.$$
\end{theorem}
\begin{proof} 
The first inequality is a consequence of the prime number theorem for arithmetic progressions and
Theorem~\ref{main}. The second inequality follows from the first one and Theorem~\ref{Theorem1}.
\end{proof}

We conjecture that $\delta_p$ exists. It is known that $\delta_3=\delta_5=\delta_7=0$ and ${\underline \delta}_{11}\ge {2\over 5}$.
We conjecture that $\delta_{11}={2\over 5}$, $\delta_{13}={1\over 3}$, $\delta_{17}={3\over 8}$, 
$\delta_{19}={4\over 9}$, $\delta_{23}={5\over 11}$ (cf.~\cite{GMW}).

\section{Kloosterman sums and their application to cyclotomic coefficients}
\label{een}
Let $p$ be a prime and, for any $\Omega\subset\RR^2$, let
\begin{equation*}
\I(\Omega):=\# \big\{
	(x,y)\in\Omega\cap \NN\times\NN:\ xy\equiv 1({\rm mod~}p)
\big\}\,.
\end{equation*}

A familiar argument using the Weil bound (\ref{weilie}) provides us with a sharp estimate 
for $\I(\Omega)$ when $\Omega$ is a rectangle:
%%%%%%%%%%%%%%%%%%%%%%%%%%%%%%%%%%
\begin{lemma}\label{Lemma1}
For any $0\le a<b<p$ and $0\le c<d<p$, let $\R:=[a,b)\times[c,d)$ or $\R:=(a,b]\times(c,d]$. 
Then we have:
\begin{equation*}
\bigg|\I\big(\R\big)-\frac{\Area(\R)}{p} \bigg|
< \sqrt p \, \big(\log p +1.1 \big)^2.
\end{equation*}
\end{lemma}
%%%%%%%%%%%%%%%%%%%%
\begin{proof}
We adapt the calculations from~\cite[Section 3.2, Lemma 4]{TezaC}.

Writing the characteristic function of the points counted by $\I(\R)$ 
in terms of exponential sums, we have:
\begin{equation}\label{NIIMOM}
\I(\R)=\frac{1}{p}\sum_{\substack{x\in(a,b]\\p\nmid x}}
\sum_{y\in(c,d]}
\sum_{k=1}^{p}
e\Big(k\frac{y-\overline{x}}{p}\Big).
\end{equation}
The main contribution is given by the terms with $k=p.$
This is equal to
\begin{equation}\label{mtmoment}
\frac{\big((b-a)+\delta_1\big)\times \big((d-c)+\delta_1\big)}{p}
=
\frac{\Area(\R)}{p}+\eta,
\end{equation}
where $|\delta_1|\le 1$, $|\delta_2|\le 1$, which implies
$|\eta|\leq 3$  for any prime $p \ge 2.$
Changing the order of summation of the remaining terms, we have
\begin{equation}\label{stea}
\frac{1}{p}\sum_{\substack{x\in(a,b]\\p\nmid x}}
\sum_{y\in(c,d]}\sum_{k=1}^{p-1}
e\Big(k\frac{y-\overline{x}}{p}\Big)=
\frac{1}{p}\sum_{k=1}^{p-1}\sum_{y\in(c,d]}
e\bigg(\frac{ky}{p}\bigg)
\sum_{\substack{x\in(a,b]\\p\nmid x}}
e\bigg(\frac{-k\overline{x}}{p}\Big)\,.
%K_{(a,b]}(0,-k,p).
\end{equation}
The most inner sum on the right-hand side of \eqref{stea} is an 
incomplete Kloosterman sum. Using a standard completion together with the upper bound 
(\ref{weilie}), yields
\begin{equation}\label{KloostermanIncomplete}
\Bigg|\sum_{\substack{x\in(a,b]\\p\nmid x}}
e\bigg(\frac{-k\overline{x}}{p}\Big)\Bigg|
\le (2+\log p)\sqrt{p}.
%K_{(a,b]}(0,-k,p).
\end{equation}
On combining  this with \eqref{NIIMOM}, \eqref{mtmoment}, and 
\eqref{stea}, we obtain:

\begin{equation*}
\begin{split}
\bigg|\I\big(\R\big)-\frac{\Area(\R)}{p} \bigg|
\leq&
\frac{1}{p}\sum_{k=1}^{p-1}\frac{2}{\big| e\big( \frac{k}{p}\big)-1\big|}
\times (2+\log p)\sqrt{p}+3\\
\leq&\frac{2+\log p}{\sqrt{p}} 
\sum_{k=1}^{p-1}\frac{1}{\sin\frac{k\pi }{p}}+3\\
\leq & \frac{2+\log p}{\sqrt{p}}\sum_{k=1}^{\frac{p-1}{2}}\frac{p}{k}+3 
\leq(1.1+\log p)^2\sqrt{p}.
\end{split}
\end{equation*}
This completes the proof of the lemma.
\end{proof}

%%%%%%%%%%%%%%%%%%%%%%%%%%%%%%%%%%%%%%%%
For a region $\Omega$ contained in $[0,1]\times[0,1]$ with
piecewise smooth boundary one can show that
\begin{equation*}
\bigg|\I\big(p\,\Omega\big)-p\Area(\Omega)\bigg|
< c(\Omega)\;p^{3/4}\log p,
\end{equation*}
for some constant $c(\Omega)$, which depends only on the region $\Omega$.
For a derivation of this result from Lemma~\ref{Lemma1}, the reader is referred to the papers of 
Laczkovich~\cite{L} and Weyl~\cite{Weyl}. In our context the regions of interest are triangles,
and in such case we can directly derive via a dyadic approximation an estimate as accurate
as the one above. Moreover, we show that $c(\text{triangle})<12$.

%%%%%%%%%%%%%%%%%
\begin{lemma}\label{Lemma2}
Let $p$ be a prime number and let
$\triangle ABC\subset [0,p-1]\times[0,p-1]$ be a right triangle with two sides parallel to the
axes of coordinates. Then
\begin{equation}\label{eqL2}
\bigg|\I\big(\triangle ABC \big)-\frac{\Area(\triangle ABC)}{p} \bigg|
< 3 \;p^{3/4}\log p.
\end{equation}
\end{lemma}

\begin{proof}
To get the lower bound, we cover dyadicly $\triangle ABC$ with rectangles $D_j^k$, as in
Figure~\ref{Figure1a}. There are $n$ diagonal rows, the $j$-th row containing $2^{j-1}$ equal
rectangles.
Thus we have
\begin{equation*}
\begin{split}
\I\big(\triangle ABC \big)\ge&
\sum_{\substack{1\le j\le n\\ \phantom{1\le \;}1\le k\le 2^{j-1}}}
\I\big(D_j^k \big).
\end{split}
\end{equation*}
Then we apply Lemma~\ref{Lemma1} for each rectangle $D_j^k$:
\begin{equation}\label{eqAhalf1}
\begin{split}
\I\big(\triangle ABC \big)\ge&
\sum_{\substack{1\le j\le n\\ \phantom{1\le \;}1\le k\le 2^{j-1}}}
\frac{\Area(D_j^k)}{p}-
\big(1+2+\cdots+2^{n-1}\big)\times  \sqrt p\, (\log p +1.1)^2.
\end{split}
\end{equation}
We denote by $T$ the area of $\triangle ABC$ and notice that $\Area\big(D_1^1\big)=T/2$, while
the size of the rectangles in row $j$ is $4$ times smaller than the size of rectangles in row
$j-1$. Then, by relation \eqref{eqAhalf1} it follows:
\begin{equation}\label{eqAhalf2}
\begin{split}
\I\big(\triangle ABC \big)\ge&
\sum_{1\le j\le n}
\frac{T}{2p}\cdot \frac{1}{4^{j-1}}\cdot 2^{j-1}-
\big(2^{n}-1\big)\times \sqrt p\,(\log p+1.1)^2\\
&>\sum_{1\le j\le n}
\frac{T}{p\,2^j}-
\big(2^{n}-1\big)\times \sqrt p\,(\log p+1.1)^2\\
&=\frac Tp -\frac{T}{p\,2^n}
-\big(2^{n}-1\big)\times \sqrt p\,(\log p+1.1)^2\\
&>\frac Tp -\frac{p}{2^{n+1}}
-\big(2^{n}-1\big)\times \sqrt p\,(\log p+1.1)^2\,,
\end{split}
\end{equation}
since $T\le p^2/2$.
We balance the last two terms taking 
%$2^n=\bigg[\sqrt{p/\big(\sqrt{p}(\log p+2)^2+2\big)}\, \bigg]+1$.
%$2^n=\big[p^{1/4}/(\log p+1.1)\, \big]+1$.
$n=\Big[\frac 14\log_2 p -\log_2\big(\sqrt{2}(\log p+1.1)\big) \Big]$.
Thus, by \eqref{eqAhalf2} it follows that there exists $c>0$ and $p_0\ge 2$, such that
\begin{equation}\label{eqAhalf3}
\begin{split}
\I\big(\triangle ABC \big)&
>\frac Tp -c\;p^{3/4}\log p\,,\quad \text{for $p\ge p_0$}.
\end{split}
\end{equation}

For the upper bound, we proceed similarly, covering completely $\triangle ABC$ with an additional
row along the diagonal, the $(n+1)$-th one, containing $2^n$ rectangles. Each of these
new rectangles are equal to those in $n$-th row.
Another way to get the upper bound is to work with the complement covering, which is the difference
between the smallest rectangle that includes $\triangle ABC$ and a series of rectangles like
those used to deduce the lower bound \eqref{eqAhalf3} (see Figure~\ref{Figure1b}).

%%%%%%%%%%%%%%%%%%%%%%%%%%%%%%%%%%
\begin{figure}[h]
 \centering
 \mbox{
 \subfigure[The inner covering. ]{
 \label{Figure1a}
    \includegraphics*[scale=0.7]{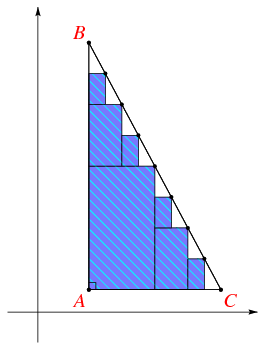}}
 }\quad\quad\quad
 \subfigure[The outer covering, by difference.]{ \label{Figure1b}
    \includegraphics[scale=0.7]{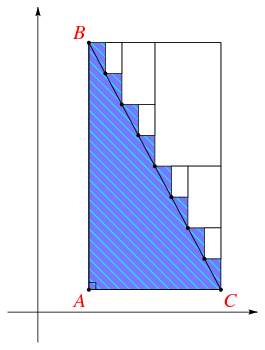}}
  \centering
 \caption{The dyadic approximations of a right triangle using three rows of rectangles.}
 \label{Figure1}
 \end{figure}
%%%%%%%%%%%%%%%%%%%%%%%%%%%%%%%%%%

We remark that a constant $c$ for which 
the left hand side of \eqref{eqL2} is less than $c\, p^{3/4}\log p$ must be larger than $2\sqrt 2$,
but 
for sufficiently large $p$, it can be chosen as close to $2\sqrt 2=2.828427125...$ as one wishes.
Numerical computations for smaller $p$ show that if $c=2.8320056$ the estimations hold
for all prime numbers $p$.
This completes the proof of the lemma.
\end{proof}

Since any $\triangle ABC\subset [0,p-1)^2$ can be obtained by starting with a rectangle whose edges are
parallel with the axes of coordinates from which at most 3 right triangles with two sides parallel
with the axes of coordinates are cut off, applying Lemma~\ref{Lemma1} and Lemma~\ref{Lemma2}, we
obtain:
%%%%%%%%%%%%%%%%%%%%%%%%%%%%%%
\begin{lemma}\label{Lemma3}
Let $p$ be a prime number and let
$\triangle ABC\subset [0,p-1]\times[0,p-1]$. Then
\begin{equation}\label{eqL3}
\bigg|\I\big(\triangle ABC \big)-\frac{\Area(\triangle ABC)}{p} \bigg|
< 12\;p^{3/4}\log p.
\end{equation}
\end{lemma}

We now  apply Lemma~\ref{Lemma3} to some special triangles. 
Let
\begin{equation}\label{eqBp}
 \B_{-}^\times(p):=
\left\{
	(x,y)\in [1,p-1]^2\cap\NN^2\colon \ 
		\begin{array}{l}  
			1\le x\le (p-3)/2,\ 
			p\le x+2y+1,\ 
			x>y,  \\ \displaystyle
		xy\equiv 1({\rm mod~}p)
		\end{array}
\right\}
\end{equation}
and 
\begin{equation}\label{eqBm}
 \B_{+}^\times(p):=
\left\{
	(x,y)\in [1,p-1]^2\cap\NN^2\colon \ 
		\begin{array}{l}  
			1\le x\le (p-3)/2,\ 
			p\le x+y,\ 
			y\le 2x,  \\ \displaystyle
		xy\equiv 1({\rm mod~}p)
		\end{array}
\right\}\,.
\end{equation}
%%%%%%%%%%%%%%%%%%%%%%%%%%%%%%%%%%
For the $52$-nd prime, $p=239$, in Figure~\ref{Figure2a} we have pictured the sets
\begin{equation*}
   \begin{split}
      B_{+}(239) = \big\{&(90, 162),(99, 169),(102, 157),(103, 181),(105, 173),(107, 172),\\
&(108, 135), (109, 182), (110, 176), (112, 207), (117, 143)\big\}
   \end{split}
\end{equation*}
and
\begin{equation*}
B_{-}(239)=  \big\{(94, 89),(95, 78),(100, 98),(101, 71),(114, 65),(115, 106),(116, 68),(118, 79)\big\}.
%|SetB-|=8 
\end{equation*}

%%%%%%%%%%%%%%%%%%%%%%%%%%%%%%%%%%
\begin{figure}[h]
 \centering
 \mbox{
 \subfigure[$\B_{+}^\times(239)$ and $\B_{-}^\times(239)$ with $11$ and $8$ points,
respectively. ]{ \label{Figure2a}
    \includegraphics[width=0.44\textwidth]{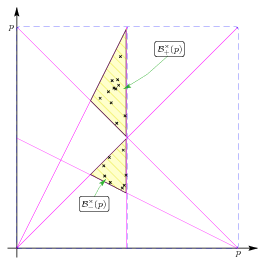}}%
 }\quad\quad
 \subfigure[$\B_{+}^\times(541)$ and $\B_{-}^\times(541)$
% with $20$ and $10$ points, respectively, 
and some catching rectangles for their extreme elements.]{ \label{Figure2b}
    \includegraphics[width=0.44\textwidth]{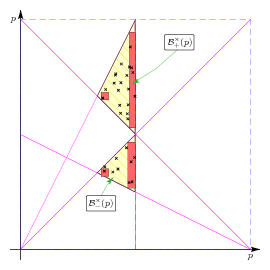}}%
  \centering
 \caption{The triangles $\B_{+}^\times(p)$ and $\B_{-}^\times(p)$.}
 \label{Figure2}
 \end{figure}
%%%%%%%%%%%%%%%%%%%%%%%%%%%%%%%%%%

\medskip

%\begin{center}
%\includegraphics[scale=1.2]{BpBm239.png}

%{\bf Figure 2. }{\small $\B_{+}(239)$ and $\B_{-}(239)$ with $11$ and $8$ points,
%respectively.}
%\end{center}
%\medskip

The sets defined by \eqref{eqBp} and \eqref{eqBm} are two disjoint triangles\footnote{
When it is clear from the context, we use the same notations $\B_{+}^{\times}(p)$ and
$\B_{+}^{\times}(p)$ not only for the lattice points, but for the triangles defined by the
inequalities on the right-hand side of \eqref{eqBp} and \eqref{eqBm}, respectively.
 }, 
and we denote their
union by $\B^\times(p):=\B_{-}^\times(p)\cup\B_{+}^\times(p)$.
Then $\B(p),\;\B_{-}(p),\; \B_{+}(p)$ are the projection onto $Ox$ of $\B^\times(p),\;\B_{-}^\times(p)$, and $\B_{+}^\times(p)$, respectively.
Notice that by projection no point is lost, as they have distinct $x$-coordinates. (This follows
since each nonzero residue class modulo $p$ has exactly one inverse modulo $p$.)

Despite some irregularities for small primes, it turns out that the number of elements in
$\B_{-}(p)$ and $\B_{+}(p)$ are approximately equal to the area of
$\B^\times(p)$ and $\B_{-}^\times(p)$, respectively, and this follows
immediately by Lemma~\ref{Lemma3}.

%%%%%%%%%%%%%%%%%%
\medskip
Let us see how far one has to go from the left or from the right side of $\B(p)$ to find
points in $\B_{-}(p)$ or $\B_{+}(p)$.
We denote 
\begin{align}
     &m_{-}(p):=\min_{x\in\B_{-}(p)}x\,,\ \
m_{+}(p):=\min_{x\in\B_{+}(p)}x\,,\
 m_{\pm}(p):=\min_{x\in\B_{-}(p)\cup\B_{+}(p)}x\,,\label{boundsm_m}\\
\intertext{and}
  &   M_{-}(p):=\max_{x\in\B_{-}(p)}x\,,\ \
M_{+}(p):=\max_{x\in\B_{+}(p)}x\,.%,\
\text{ and }\ M_{\pm}(p):=\max_{x\in\B_{-}(p)\cup\B_{+}(p)}x\,.\label{boundsm_M}
\end{align}
In case any of the sets involved is empty we set the corresponding quantity to be
$p/2$ if a `max' is involved, and $p/3$ is a `min' is involved. Thus if $\B_{-}(p)$ is
empty, then $m_{-}(p)=p/3$, for example.

The known methods to study the distribution of inverses ultimately reduce to
showing the existence of small boxes $\R=\I\times\J\subset [1,p]^2$ that capture points 
$(x, \overline{x})$.
In our case the problems are not the same at both ends. This is due to the vertical
edges that exist only on the right-hand side of $\B_{-}^\times(p)$ and $\B_{+}^\times(p)$ 
(see Figure~\ref{Figure2}). 
What we are looking for is a slim box $\R=\I\times\J\subset\B(p)$ that contains
elements of $\B_{-}^\times(p)$ or $\B_{+}^\times(p)$, has the edge $\I$ as small as possible (which
by Lemma~\ref{Lemma1} means that the length of the other edge  $\J$ is forced to be as large as
possible), and is situated as
close as possible to the left or to the right of $\B(p)$, respectively.

Lemma~\ref{Lemma1} shows that the points counted by $\I(\R)$ are rather uniformly spread out
into $[0,p]^2$, therefore it suffices to allow us to find the smallest rectangles $\R\subset
[0,p]^2$ for
which we know for sure that $\I(\R)$ is positive. The condition is:
\begin{equation*}
0 \le  \frac{\Area(\R)}{p} - \sqrt p(\log p +1.1)^2< \I\big(\R\big)\,,
\end{equation*}
which becomes
\begin{equation}\label{eqcond}
  p^{3/2}(\log p +1.1)^2\le \Area(\R)\,.
\end{equation}

Condition \eqref{eqcond} and inclusion in 
$\B_{-}(p)$, $\B_{+}(p)$ or $\B(p)$
are the only two requirements that  our capturing
boxes must fulfill.  These imply sharper estimates for $M_{-}(p)$,  $M_{+}(p)$ and
$M_{\pm}(p)$ than for the corresponding ones on the left-hand side. The reason is that near $x=p/2$
the edges of the triangles from Figure~\ref{Figure2} are long, so we can afford to take $\J$ with 
$|\J|=O(p)$.
On the other hand, for the bound of  $m_{-}(p)$,  $m_{+}(p)$ and
$m_{\pm}(p)$
we can not due better than fit approximately square boxes (rectangles
with edges of the
same order of magnitude), because of the slopes of the edges of the triangles $\B_{-}(p)$,
$\B_{+}(p)$ that meet at  $x=p/3$ (see Figure~\ref{Figure2b}).

The following theorem gives the estimates that follow for the quantities defined in
\eqref{boundsm_m} and \eqref{boundsm_M}. Numerically we found that, for $p < 10^8$,
$m_-(p) \le p/3 + 6 \sqrt{p}$, $m_+(p) \le p/3 + 4\sqrt{p}$ and that
for $p < 10^{10}$, $p/2 - 4 \log p \le M_-(p)$, $p/2 - 2 \log p \le M_+(p)$. 

%%%%%%%%%%%%%%%%%%%%%%%%%%%%%%%%%%%%%
\begin{theorem}\label{Theorem2}
%Let $\B$ be any of $\B_{-}(p)$, $\B_{+}(p)$, $\B(p)$. Then
%\begin{align}
% p/3\le&m\le p/3 + 3 p^{3/4}(\log p +2)/2\\
%\intertext{and}
 % p/2 - p^{3/4}(\log p +2)\le &M\le p/2\,.
%\end{align}
For $p\ge 2$ we have:
%\begin{equation}\label{eqmM}
%\begin{split}
%m&=p/3+O(p^{3/4}\log p)\,,\\
%M&=p/2+O(p^{1/2}\log^2 p)\,.\\
%\end{split}
%\end{equation}
\begin{align}
     \frac{p}{3}\le &m_{-}(p)\le \frac{p}{3}+4.25\, p^{3/4}\log p \\
     \frac{p}{3}\le &m_{+}(p)\le \frac{p}{3}+3\, p^{3/4}\log p\\
     \frac{p}{3}\le &m_{\pm}(p)\le \frac{p}{3}+3\, p^{3/4}\log p \label{eqmpm}\\
\intertext{and}
     \frac{p}{2}-3\, p^{1/2}\log^2 p\le &M_{+}(p)\le \frac{p}{2}\label{eqqui}\\
     \frac{p}{2}-6\, p^{1/2}\log^2 p\le &M_{-}(p)\le \frac{p}{2}\label{eqMpM}
\end{align}
Moreover, for $p\ge 11$:
\begin{align*}
\text{ if } &p\equiv 1({\rm mod~}3), \text{ then } M_{+}(p)=(p-3)/2\,,\\ 
\text{ if } &p\equiv 2({\rm mod~}3), \text{ then } M_{-}(p)=(p-3)/2\,,\\
\intertext{and}
& M_{\pm}(p)=\max\big\{M_{-}(p),\,M_{+}(p)\big\}=(p-3)/2\,. 
\end{align*}
%Let $p\ge 11$.\\ 
%If $p\equiv 1({\rm mod~}3)$, then $M_{+}(p)=(p-3)/2$.\\ 
%If $p\equiv 2({\rm mod~}3)$, then $M_{-}(p)=(p-3)/2$.\\
%Further, $M_{\pm}(p)=\max\big\{M_{-}(p),\,M_{+}(p)\big\}=(p-3)/2$. 
\end{theorem}
\begin{proof} When dealing with any of the six quantities, we may assume that the associated $\B$ set is non-empty, for
if it is empty the inequality to be proved trivially holds true.\\
\indent First we find the upper bound of $m_{+}(p)$.
Let $\R\subset \B_{+}(p)$ be the capturing box from the left side of in Figure~\ref{Figure2b}. We
denote
its height by $H$, and assume it to be the largest possible. Also, let $L$ be the length of the
horizontal edge of $\R$, and let $l$ be the distance from the left edge of $\R$ to
$x=p/3$. 
By the similarity of triangles, it follows that 
$H/(p/2)=l/(p/6)$, that is, $l=H/3$.
Denoting $\alpha:=H/L$, this can be written as
\begin{equation}\label{eqlaL1}
     l=\frac{\alpha}{3}L\,.
\end{equation}

Putting $b(p):=p^{3/2}(\log p +1.1)^2$, the inequality \eqref{eqcond} becomes
\begin{equation}\label{eqlaL2}
    b(p)\le \alpha L^2\,.
\end{equation}

Let us remark that if $\R$ is a box that contains points from $\B_{+}^{\times}(p)$, then 
$m_{+}(p)\le p/3+ l+L$. Then, because we need the best available bound, using \eqref{eqlaL1} and
\eqref{eqlaL2}, we get:
\begin{equation*}   
m_{+}(p)-\frac p3\le\min_{b(p)\le \alpha L^2}(l+L)
\le \min_{\frac{b(p)}{L}\le \alpha L}\left(\frac{\alpha L}{3}+L\right)
=\min_{L}\left(\frac{b(p)}{3L}+L\right),
\end{equation*}
where we have made the choice $\alpha=b(p)L^{-2}$, that is $HL=b(p)$.
We balance the terms here, taking $L=\sqrt{b(p)/3}$. These yield
\begin{equation*}
     \begin{split}
          m_{+}(p)-\frac p3 \le\frac{2\sqrt{b(p)}}{\sqrt{3}}\le c_{+}p^{3/4}\log p\,,
     \end{split}
\end{equation*}
for some positive constant $c_{+}$. For sufficiently large $p$, we can take $c_{+}$ close to
$2/\sqrt 3$, while  $c_{+}=3$ covers the inequality for all $p\ge 2$.

The bound for $m_{-}(p)$ is obtained in a similar way. In this case 
$H/(p/4)=l/(p/6)$, and equality \eqref{eqlaL1} has to be replaced by $l=2\alpha L/3$. Then, the same reasoning (with 
$b(p)$ replaced by $2b(p)$) gives
\begin{equation*}
     \begin{split}
          m_{-}(p)-\frac p3 \le\frac{2\sqrt{2b(p)}}{\sqrt{3}}\le c_{-}p^{3/4}\log p\,,
     \end{split}
\end{equation*}
where  $c_{-}=\sqrt 2\, c_{+}$. To cover the bound for all $p\ge 2$, it suffices to take 
$c_{-}=4.25$.

On noticing that
\begin{equation*}
	m_{\pm}(p)=\min\big\{m_{-}(p),\,m_{+}(p)\big\}\,,    
\end{equation*}
the estimate \eqref{eqmpm} follows.

Now we focus on the other side of the triangle and consider  a rectangle $\R\subset \B_{+}^{\times}(p)$ with one edge
glued on the right edge of $\B_{+}^{\times}(p)$. The length of the horizontal edge of $\R$ is $L$
and the length of the vertical one is $H$. As before, we assume that $H$ is as large as possible.

Then, by the similarity of triangles, it follows that 
$H/(p/2)=(p/6-L)/(p/6)$, that is, $H=(p-6L)/2$.
Then, the inequality \eqref{eqcond} becomes
\begin{equation}\label{eqlaL22}
    2\,b(p)\le PL- 6L^2\,.
\end{equation}
We need to find the smallest $L$ for which \eqref{eqlaL22} is satisfied, since
\begin{equation*}
    \frac{p}{2}-M_{+}\le \min_{2b(p)\le pL-6L^2}L\,.
\end{equation*}
Such an $L$ gives rise to the estimate
\begin{equation*}
 \frac{p}{2} -  M_{+}\le C_{+}\sqrt{p}\,\log^2 p\,,
\end{equation*}
for some $C_{+}>2,$ but it can be chosen infinitely close to $2$ for all $p>p_{C_+}$. 

The analogous estimate for $M_-$ is obtained similarly in the other triangle $\B_{-}^{\times}$,
and we get
\begin{equation*}
 \frac{p}{2} -   M_{-}\le C_{-}\sqrt{p}\,\log^2 p\,, \quad\text{ for $p\ge p_{C_{-}}$\,. }
\end{equation*}
Moreover, we get $C_{-}=2\,C_{+}$ and $p_{C_{-}}=2\,p_{C_{+}}$.

For instance, we may take get $C_{+}=3$ and $p_{C_{+}}=8.6\times 10^8$, but one may 
establish variants of these estimates, tightening up or down both the constants and/or the
domain on which they are fulfilled.

By direct computation one then checks that the inequalities (\ref{eqqui}) and (\ref{eqMpM}) are satisfied for
every prime $2\le p\le 8.6\time 10^8$.

The remaining part of the result follows from the proof of Lemma~\ref{uno}.\end{proof}

We remark that the exponents $3/4$ and $1/2$ are essentially the smallest that can 
derived by this method.
How much further can they be decreased? Let $\R=\I\times\J\subset [1,p]^2$ be a rectangle. 
Arguing probabilistically, if $p$ is large, for any $x\in[1,p-1]$, the probability that
$\overline{x} \in\J$ is $\sim |\J|/p$. Then the probability that there exists a point with
coordinates $(x,\overline{x})\in\R$ should be $\sim\Area(\R)/p$. This leads us to conjecture that 
\begin{conjecture}\label{Conjecture1} Let $\epsilon>0$. We have
\begin{equation}\label{eqConjecture}
\begin{split}
m&={p\over 3}+O(p^{1/2+\epsilon})\,,\\
M&={p\over 2}+O(p^{\epsilon})\,.\\
\end{split}
\end{equation}
\end{conjecture}

We claim that the exponent $1/2$ on the right side of (\ref{eqConjecture})
is best possible. Indeed, assume for instance that $p\equiv 1({\rm mod~}3)$. 
If $(x,y)\in\B_{-}^\times(p)$, write $x=\frac{p-1}3+a$, $y=\frac{p-1}3+b$,
and then from $xy\equiv 1({\rm mod~}p)$ it follows that
$(3a-1)(3b-1)\equiv 9 ({\rm mod~}p)$. We cannot have $(3a-1)(3b-1)=9$, therefore
$|(3a-1)(3b-1)-9|\ge p$, and since $|b| \ll a$, we deduce that $a\gg \sqrt{p}$.

%%%%%%%%%%%%%%%%%%%%%%%%%%%%%%%%%%%%%%%%%%%
\section{A sharper lower bound for $M(p)$ valid for an infinite set of primes}
\label{twee}
In this section we show that the inequality \eqref{e4} holds for an infinite class of prime numbers
$p$ and we establish Theorem~\ref{driedrie}. For our 
construction to work we need an improvement over the well known 
Bombieri-Vinogradov Theorem. We may arrange the proof so that we work
with a fixed residue class, and in such case a strong improvement over the
Bombieri-Vinogradov Theorem has been provided in a series of papers by
Bombieri, Friedlander and Iwaniec~\cite{BFI1, BFI2, BFI3}.
The Main Theorem from~\cite{BFI3} gives a continuous transition from
Bombieri-Vinogradov type theorems to Brun-Titchmarsh type theorems. 
It states that:
%%%%%%%%%%%%%%%%%%%%%
\begin{theorem}[Bombieri-Friedlander-Iwaniec~\cite{BFI3}]\label{ThBFI}
     Let $a \neq 0$ be an integer and $A > O$, $2\le  Q\le x^{3/4}$ be reals.
Let $\C$ be the set of all integers $q$, prime to  $a$, from an interval $Q' < q \le Q$.
Then
\begin{equation}\label{eqThBFI}
     \begin{split}
	\sum_{q\in\C}&\left|\pi (x;q ,a)-\frac{\pi(x)}{\varphi(q)}
\right|\\
&\le
\left\{K\Big(\theta - \frac 12\Big)^2 {x\over L}+
O_{A}\Big({x\over L^3}\big(\log\log x\big)^2\Big)
\right\}
\sum_{q\in\C}\frac{1}{\varphi(q)}
+
O_{a,A}\Big({x\over L^{A}}\Big),
\end{split}
\end{equation}
where $\theta  = \log Q/ \log x$,    $L = \log x$,  $K$ is absolute, and the subscripts of $O$
indicate the dependence on those constants.
\end{theorem}
%%%%%%%%%%%%%%%%%%%%%

Fix a constant $c_3>1$, and two other constants $0<c_4<c_5$. 
Take a large positive real number $X$ and apply the above estimate
with $a=-9$, $A = 3$, $x=X$, $Q = c_5\sqrt{X}$, and $Q' = c_4\sqrt{X}$.
Then $L = \log X$ and
$\theta = \log Q/\log X = 1/2+\log c_5/\log X,$ so
\begin{equation*}
\left(\theta - \frac12\right)^2 =  \frac{\log^2 c_5}{\log^2X}\,.
\end{equation*}
%%%%%%%%%%%%%%%%%%
Landau~\cite[p. 113]{Edmund} showed that
\begin{equation*}
\sum_{n\le x}{1\over \varphi(n)}=\alpha \log x + \beta +O\left({\log x\over x}\right),
\end{equation*}
with $\alpha>0$ and $\beta$ constants that can be explicitly given.
This implies
\begin{equation*}
\sum_{\substack{Q'<q<Q\\(q,3)=1}} {1\over \varphi(q)}  = 
O\left(\sum_{Q'<q<Q} {1\over \varphi(q)}\right)  
= O(1).
\end{equation*}
It follows that
\begin{equation}\label{Ye1}
\sum_{\substack{c_4\sqrt{X}<q<c_5\sqrt{X}\\(q,3)=1}}\left|\pi(X;q,-9)-\frac{\pi(X)}{\varphi(q)}
\right|
=O\left(\frac{X(\log\log X)^2}{\log^3 X}\right)\,.
\end{equation}
Applying the estimate a second time, with $a=-9$, $A = 3$, $x=c_3X$, $Q = c_5\sqrt{X}$, and 
$Q' = c_4\sqrt{X}$, we have
\begin{equation}\label{Ye2}
\sum_{\substack{c_4\sqrt{X}<q<c_5\sqrt{X}\\(q,3)=1}}\left|\pi(c_3X;q,-9)-\frac{\pi(c_3X)}{\varphi(q)
} \right|
=O\left(\frac{X(\log\log X)^2}{\log^3 X}\right)\,.
\end{equation}

Next, we restrict the summation over $q$ on the left sides of \eqref{Ye1} and \eqref{Ye2}
to prime numbers congruent to $-1({\rm mod~}3)$, and then combine the two estimates to
obtain
\begin{equation}\label{Ye3}
\sum_{\substack{q \text{ prime}\\
c_4\sqrt{X}<q <c_5\sqrt{X}\\  q\equiv -1({\rm mod~}3)}}
\left|\pi(c_3X;q,-9) - \pi(X;q,-9)-\frac{\pi(c_3X) - \pi(X)}{q-1}\right|
=O\left(\frac{X(\log\log X)^2}{\log^3 X}\right)\,.
\end{equation}

Furthermore, 
\begin{equation}\label{Ye4}
\sum_{\substack{q \text{ prime}\\
c_4\sqrt{X}<q <c_5\sqrt{X}\\  q\equiv -1({\rm mod~}3)}}
\frac{1}{q}\sim \frac{\log c_5 - \log c_4}{\log X} \;,
\end{equation}
and
\begin{equation}\label{Ye5}
\sum_{\substack{q \text{ prime}\\
c_4\sqrt{X}<q <c_5\sqrt{X}\\  q\equiv -1({\rm mod~}3)}}
\frac{\pi(c_3X) - \pi(X)}{q-1}\sim \frac{(c_3-1)(\log c_5 - \log c_4)X}{\log^2 X}\;.
\end{equation}

Combining \eqref{Ye3} and \eqref{Ye5}, we find that
\begin{equation}\label{Ye6}
\sum_{\substack{q \text{ prime}\\
c_4\sqrt{X}<q <c_5\sqrt{X}\\  q\equiv -1({\rm mod~}3)}}
\left(\pi(c_3X;q,-9) - \pi(X;q,-9)\right)\sim \frac{(c_3-1)(\log c_5 - \log c_4)X}{\log^2 X}\;.
\end{equation}

Let us remark that for each prime number $p\le c_3X$, there are at most two
prime numbers $q\in(c_4\sqrt{X}, c_5\sqrt{X})$ for which $p\equiv -9({\rm mod~}q)$,
so each prime $p$ is counted at most twice on the left side of \eqref{Ye6}. 
We deduce that
%\begin{equation}\label{Ye6}
%\begin{split}
%&\#\Big\{p \; \text{prime}: X < p < c_3X,\  p\equiv -9\pmod q, \; \text{for some prime}
%\;q\;
%\text{with}\\
%&\phantom{\#\Big\{p \; \text{prime}: : X < p < c_3X, }
%c_4\sqrt{X} < q < c_5\sqrt{X}, q\equiv -1\pmod 3 \Big\}\\
% \ge &\frac{c_6X}{\log^2 X}\;,
%     \end{split}
%\end{equation}

\begin{equation}\label{Ye67}
\#\left\{p \ :
\begin{array}{l l l}p \text{ prime, }X < p < c_3X,\  \\
p\equiv -9({\rm mod~}q), \; \text{ for some prime } q \text{ with }\\
\qquad c_4\sqrt{X} < q < c_5\sqrt{X} \text{ and } q\equiv -1({\rm mod~}3)     
\end{array}
 \right\}
\ge \frac{c_6X}{\log^2 X}\;,
\end{equation}
for any fixed real number $c_6$, satisfying
\begin{equation}\label{Ye7}
0< c_6 < \frac{(c_3-1)(\log c_5 - \log c_4)}{2}\;,
\end{equation}
and all $X$ large enough.

We now take any prime $p$ from the set on the left side of \eqref{Ye67}, choose
a corresponding $q$, and write $p+9 = qm$.  We distinguish two cases.

{\bf (I)}  $p\equiv -1({\rm mod~}3)$.
In this case we have $m\equiv 1({\rm mod~}3)$. We write $q$ and $m$ in the form
$q=3a-1$, $m=3b+1$. Here $a$ and $b$ are positive integers, and each of them 
lies between two (suitable) constants times $\sqrt{X}$. We put $y = \frac{2p-1}3 + a$,
$x=\frac{p+1}3+b$. Then $x$ and $y$ are integers and satisfy the congruence
$xy\equiv 1({\rm mod~}p)$. The point $(x,y)$ lies (for suitably chosen constants
$c_3, c_4$ and $c_5$) inside the upper yellow triangle in Figure~\ref{Figure2},
close to its left vertex.

{\bf (II)}  $p\equiv 1({\rm mod~}3)$.
In this case $m\equiv -1({\rm mod~}3)$. Write 
$q=3a-1$, $m=3b-1$. As before, $a$ and $b$ are positive integers and each 
lies between two constants times $\sqrt{X}$. We now put $x = \frac{p-1}3 + a$,
$y=\frac{p-1}3+b$. Then $x$ and $y$ are integers satisfying
$xy\equiv 1({\rm mod~}p)$. Moreover, one of the points $(x,y)$ or $(y,x)$
lies (for suitably chosen 
$c_3, c_4$ and $c_5$) inside the lower yellow triangle (shaded triangle in 
a black-white rendition of this article) in Figure~\ref{Figure2},
close to its left vertex.

Putting both cases together, we see that for such prime numbers $p$, the first equality in
Conjecture~\ref{Conjecture1} holds in the stronger form $m= p/3 + O(\sqrt{p})$. This $m$ here 
is the one defined via the union of the two yellow triangles in Figure~\ref{Figure2a}
(and~\ref{Figure2b}), and the implied constant is effectively computable. 

In conclusion, we have proved the following theorem.

%%%%%%%%%%%%%%%%%%%%%%%%%
\begin{theorem}\label{Theorem4}
For all large $X$, we have
$\# \{p\le X: M(p) >  2p/3 - c_8 \sqrt{p}\}\ge {c_7X}{\log^{-2} X}.$
\end{theorem}
%%%%%%%%%%%%%%%%%%%%%%%%%

Finally, we will establish Theorem~\ref{driedrie} stated in the introduction.\\

{\it Proof of Theorem}~\ref{driedrie}. 
The estimate~\eqref{e3} is a consequence of Theorem~\ref{Theorem2} and 
the inequality $M(p)\ge p-m_{\pm}(p)$, which follows by Corollary~\ref{coo}. 
Part 2 is a corollary of Theorem~\ref{Theorem4}. \qed

%\vfil\eject
%%%%%%%%%%%%%%%%%%%%%%%%%%%%%%%%%%%%%
%%%%%%%%%%%%%%%%%%%%%%%%%%%%%%%%%%%%%

%%%%%%%%%%%%%%%%%%%%%%%%%%%%%%%%%%%%%%%%%5
%%%%%%%%%%%%%%%%%%%%%%%%%%%%%%%%%%%%%%%%%%%%%
%%%%%%%%%%%%%%%%%%%%%%%%%%%%%%%%%%%%%%%%%%%%
%%%%%%%%%%%%%%%%%%%%%%%%%%%%%%%%%%%%%%%%%
%%%%%%%%%%%%%%%%%%%%%%%%%%%%%%%%%%%%%%%%%%%%

%%%%%%%%%%%%%%%%%%%%%%%%%%%%%%%%%%%%%%%

\end{document}